\documentclass[10pt]{article}

\usepackage{amsxtra,amssymb,amsthm,amsmath,latexsym}

\usepackage{graphicx}
\newtheorem{thm}{Theorem}[section]

\newtheorem{lem}[thm]{Lemma}
 \newcommand{\thmref}[1]{Theorem~\ref{#1}}
 \newcommand{\lemref}[1]{Lemma~\ref{#1}}

\newcommand{\R}{{\mathbb R}}
\newcommand{\N}{{\mathbb N}}
\newcommand{\dl}{{\delta}}
\newcommand{\bee}{\begin{equation*}}
\newcommand{\eee}{\end{equation*}}
\newcommand{\be}{\begin{equation}}
\newcommand{\ee}{\end{equation}}
\newcommand{\pn}{\par\noindent}
\title{
Dynamical Systems Method for Solving Ill-conditioned Linear
Algebraic Systems }
\author{Sapto W. Indratno\\
\small Department of Mathematics\\[-0.8ex]
\small Kansas State University, Manhattan, KS 66506-2602, USA\\
\small \texttt{sapto@math.ksu.edu}\\
\and
A G Ramm\\
\small Department of Mathematics\\[-0.8ex]
\small Kansas State University, Manhattan, KS 66506-2602, USA\\[-0.8ex]
\small \texttt{ramm@math.ksu.edu}\\
}
\date{}
\begin{document}

\maketitle

\begin{abstract}
A new method, the Dynamical Systems Method (DSM), justified
recently, is applied to solving ill-conditioned linear algebraic
system (ICLAS). The DSM gives a new approach to solving a wide class
of ill-posed problems. In this paper a new iterative scheme for
solving ICLAS is proposed. This iterative scheme is based on the DSM
solution. An a posteriori stopping rules for the proposed method is
justified. This paper also gives an a posteriori stopping rule for a
modified iterative scheme developed in A.G.Ramm, JMAA,330
(2007),1338-1346, and proves
convergence of the solution obtained by the iterative scheme.
\end{abstract}
\pn{

 }\\

\noindent\textbf{Biographical notes:} Professor Alexander G. Ramm is
an author of more than 580 papers, 2 patents, 12 monographs, an
editor of 3 books, and an associate editor of several mathematics
and computational mathematics Journals. He gave more than 135
addresses at various Conferences, visited many Universities in
Europe, Africa, America, Asia, and Australia. He won Khwarizmi Award
in Mathematics, was Mercator Professor, Distinguished Visiting
Professor supported by the Royal Academy of Engineering, invited
plenary speaker at the Seventh PanAfrican Congress of
Mathematicians, a London Mathematical Society speaker, distinguished
HKSTAM speaker, CNRS research professor, Fulbright professor in
Israel, distinguished Foreign professor in Mexico and Egypt. His
research interests include inverse and ill-posed problems,
scattering theory, wave propagation, mathematical physics,
differential and integral equations, functional analysis, nonlinear
analysis, theoretical numerical analysis, signal processing, applied
mathematics and operator theory.\\

\noindent Sapto W. Indratno is currently a PhD student at Kansas
State University under the supervision of Prof. Alexander G. Ramm.
He is a coauthor of three accepted papers. His fields of interest
are numerical analysis, optimization, stochastic processes, inverse
and ill-posed problems, scattering theory, differential equations
and applied mathematics.

\section{Introduction}
We consider a linear equation
\begin{equation}\label{11}
Au=f,
\end{equation}
where $A:\R^m\to \R^m,$ and assume that equation (\ref{11}) has a
solution, possibly non-unique. According to Hadamard
\cite[p.9]{RAMM07}, problem (\ref{11}) is called well-posed if the
operator $A$ is injective, surjective, and $A^{-1}$ is continuous.
Problem (\ref{11}) is called ill-posed if it is not well-posed.
Ill-conditioned
linear algebraic systems arise as discretizations of
ill-posed problems, such as Fredholm integral equations of the first
kind, \be\label{Fred1} \int_a^b k(x,t)u(t)dt=f(x),\ c\leq x\leq d,
\ee where $k(x,t)$ is a
smooth kernel. Therefore, it is of interest to develop a method for solving ill-conditioned linear algebraic systems stably.
  In this paper we give a method for solving
linear algebraic systems \eqref{11} with an ill-conditioned-matrix
$A$. The matrix $A$ is called ill-conditioned if $\kappa(A)>>1,$
where $\kappa(A):=||A|||A^{-1}||$ is the condition number of $A$. If
the null-space of $A$, $ \mathcal{N}(A):=\{u: Au=0\}$, is
non-trivial, then $\kappa(A)=\infty$. Let $A=U\Sigma V^*$ be the
singular value decomposition (SVD) of $A$, $UU^*=U^*U=I$,
$VV^*=V^*V=I$, and
$\Sigma=\text{diag}(\sigma_1,\sigma_2,\hdots,\sigma_m)$, where
$\sigma_1\geq \sigma_2\geq \dots\geq \sigma_m\geq 0$ are the singular
values of $A.$ Applying this SVD to the matrix $A$ in (\ref{11}),
one gets \be\label{LSP}f=\sum_i\beta_i u_i\text{ and
}y=\sum_{i,\sigma_i>0}
\frac{\beta_i}{\sigma_i}v_i,\ee where $\beta_i=\langle
u_i,f\rangle$. Here $\langle \cdot,\cdot\rangle$ denotes the inner
product of two vectors. The terms with small singular values
$\sigma_i$ in (\ref{LSP}) cause instability of the solution, because
the coefficients $\beta_i$ are known with errors. This difficulty is
essential  when one deals with an ill-conditioned matrix $A$.
Therefore a regularization is needed for solving ill-conditioned
linear algebraic system (\ref{11}). There are many methods to solve
\eqref{11} stably: variational regularization, quasisolutions,
iterative regularization (see e.g, \cite{NHRAMM08}, \cite{MRZ84},
\cite{RAMM05}, \cite{RAMM07}). The method proposed in this paper is
based on the Dynamical Systems Method (DSM) developed in
\cite[p.76]{RAMM07}. The DSM for solving equation (\ref{11})
 with, possibly, nonlinear operator $A$ consists of solving the Cauchy
problem
\begin{equation}\label{dsm1}
\dot{u}(t)=\Phi(t,u(t)),\quad  u(0)=u_0;\
\dot{u}(t):=\frac{du}{dt},
\end{equation} where $u_0\in H$ is an arbitrary element of a Hilbert space $H$, and $\Phi$ is some nonlinearity, chosen so that the following three
conditions hold: a) there exists a unique solution
$u(t)\quad\forall t\geq 0,$  b) there exists $u(\infty)$, and c)
$Au(\infty)=f.$

In this paper we choose $\Phi(t,u(t))=(A^*A+a(t)I)^{-1}f-u(t)$ and
consider the following Cauchy problem:  \be\label{DSM3}
\dot{u}_a(t)=-u_a(t)+\left[A^*A+a(t)I_m\right]^{-1}A^*f,\quad
u_a(0)=u_0, \ee where \be\label{cat} a(t)>0,\text{ and }a(t)\searrow
0\text{ as }t\to \infty ,\ee $A^*$ is the adjoint matrix and $I_n$
is an $m\times m$ identity matrix. The initial element $u_0$ in
(\ref{DSM3}) can be chosen arbitrarily in $N(A)^\perp$, where
\be\label{NA} \mathcal{N}(A):=\{u\ | \ Au=0\}.\ee For example, one
may take $u_0=0$ in \eqref{DSM3} and then the unique solution to
(\ref{DSM3}) with $u(0)=0$ has the form \be\label{DSMsol}
u(t)=\int_0^{t}e^{-(t-s)}T_{a(s)}^{-1} A^*f ds,\ee where $T:=A^*A$,
$T_a:=T+aI$, $I$ is the identity operator. In the case of noisy data
we replace the exact data $f$ with the noisy data $f_\dl$ in
\eqref{DSMsol}, i.e., \be\label{DSMsol2}
u^\dl(t)=\int_0^{t_\dl}e^{-(t_\dl-s)}T_{a(s)}^{-1} A^*f_\dl ds,\ee
where $t_\delta$ is the stopping time which will be discussed later.
There are many ways to solve the Cauchy problem (\ref{DSM3}). For
example, one may apply a family of Runge-Kutta methods for solving
(\ref{DSM3}). Numerically, the Runge-Kutta methods require an
appropriate stepsize to get an accurate and stable solution. Usually
the stepsizes have to be chosen sufficiently small to get such a
solution. The number of steps will increase when $t_\delta$, the
stopping time, increases, see \cite{NHRAMM08}. Therefore the
computation time will increase significantly. Since $\lim_{\delta\to
0}t_\delta = \infty$, as was proved in \cite{RAMM07}, the family of
the Runge-Kutta method may be less efficient for solving the Cauchy
problem (\ref{DSM3}) than the method, proposed in this paper. We
give a simple iterative scheme, based on the DSM, which produces
stable solution to equation \eqref{11}. The novel points in our
paper are iterative schemes \eqref{itnf1} and \eqref{itn1} (see
below), which are constructed on the basis of formulas
\eqref{DSMsol} and \eqref{DSMsol2}, and a modification of the
iterative scheme given in \cite{RAMM504}. Our stopping rule for the
iterative scheme \eqref{itn1} is given in \eqref{srule} (see below).
In \cite[p.76]{RAMM07} the function $a(t)$ is assumed to be a slowly
decaying monotone function. In this paper instead of using the
slowly decaying continuous function $a(t)$ we use the following
piecewise-constant function: \be\label{ae}
a^{(n)}(t)=\sum_{j=0}^{n-1}
\alpha_0q^{j+1}\chi_{(t_j,t_{j+1}]}(t),\quad q\in(0,1),\quad
t_j=-j\ln (q),\ n\in\N, \ee where $\N$ is the set of positive
integer, $t_0=0$, $\alpha_0>0$, and
\be\label{chi1}\chi_{(t_j,t_{j+1}]}(t)=\left\{
                         \begin{array}{ll}
                           1, & \hbox{$t\in (t_j,t_{j+1}]$;} \\
                           0, & \hbox{otherwise.}
                         \end{array}
                       \right.\ee The parameter $\alpha_0$ in \eqref{ae} is chosen so that assumption
\eqref{Asfd} (see below) holds.
This assumption plays an important role in the
proposed iterative scheme. Definition \eqref{ae} allows one to write
\eqref{DSMsol} in the form \be\begin{split}\label{itnf1}
u_{n+1}&=qu_{n}+(1-q)T_{\alpha_0q^{n+1}}^{-1}A^*f,\quad u_0=0.
\end{split}\ee A detailed derivation of the iterative scheme
\eqref{itnf1} is given in Section 2. When the data $f$ are
contaminated by some noise, we use $f_\dl$ in place of $f$ in
\eqref{DSMsol}, and get the iterative scheme
\be\label{itn1}\begin{split}
u_{n+1}^\dl&=qu_n^\dl+(1-q)T_{\alpha_0q^{n+1}}^{-1}A^*f_\dl,\quad
u_0^\dl=0.
\end{split}\ee We always assume that\be\label{fdf}
\|f_\dl-f\|\leq \dl, \ee where $f_\dl$ are the noisy data, which are
known, while $f$ is unknown, and $\dl$ is the level of noise. Here
and throughout this paper the notation $\|z\|$ denotes the
$l^2$-norm of the vector $z\in \R^m$. In this paper a discrepancy
type principle (DP) is proposed to choose the stopping index of
iteration \eqref{itn1}. This DP is based on discrepancy principle
for the DSM developed in \cite{RAMM525}, where the stopping time
$t_\dl$ is obtained by solving the following nonlinear equation
\be\label{MDP}
\int_0^{t_\dl}e^{-(t_\dl-s)}a(s)\|Q_{a(s)}^{-1}f_\dl\|ds=C\dl,\quad
C\in(1,2]. \ee It is a non-trivial task to obtain the stopping time
$t_\dl$ satisfying \eqref{MDP}. In this paper we propose a
discrepancy type principle based on \eqref{MDP} which can be easily
implemented numerically: iterative scheme \eqref{itn1} is stopped at
the first integer $n_\dl$ satisfying the inequalities:
\be\label{rule}\begin{split}
&\sum_{j=0}^{n_\dl-1}(q^{n_\dl-j-1}-q^{n-j})\alpha_0q^{j+1}\|Q_{\alpha_0q^{j+1}}^{-1}f_\dl\|
ds\leq C\dl^\varepsilon\\
&<
\sum_{j=0}^{n-1}(q^{n-j-1}-q^{n-j})\alpha_0q^{j+1}\|Q_{\alpha_0q^{j+1}}^{-1}f_\dl\|
,\ 1\leq n<n_\dl, \end{split}\ee and it is assumed that
\be\label{Asfd} (1-q)\alpha_0q\|Q_{\alpha_0q}^{-1}f_\dl\|\geq
C\dl^\varepsilon,\quad C>1,\quad\varepsilon\in(0,1),\quad
\alpha_0>0.\ee We prove in Section 2 that using discrepancy-type
principle \eqref{rule}, one gets the convergence: \be \lim_{\dl\to
0}\|u_{n_\dl}^\dl-y\|=0, \ee where $u_n^\dl$ is defined in
\eqref{itn1}. About other versions of discrepancy principles for DSM
we refer the reader to \cite{RAMM05},\cite{RAMM480}. In this paper
we assume that $A$ is bounded. If the operator $A$ is unbounded then
$f_\dl$ may not belong to the domain of $A^*$. In this case the
expression $A^*f_\dl$ is not defined. In \cite{RAMM500},
\cite{RAMM504} and \cite{RAMM522} solving \eqref{11} with unbounded
operators is discussed. In these papers the unbounded operator $A$
is assumed to be linear, closed, densely defined operator in a
Hilbert space. Under these assumptions one may use the operator
$A^*(AA^*+aI)^{-1}$ in place of $T_a^{-1}A^*$. This operator is
defined for any $f$ in the Hilbert space. \\ In \cite{RAMM504} an
iterative scheme with a constant regularization parameter is given:
\be\label{itr1}\ u_{n+1}^\dl=aT_a^{-1}u_n^\dl+T_a^{-1}A^*f_\dl, \ee
but the stopping rule, which produces a stable solution of equation
\eqref{11} by this iterative scheme, has not been discussed in
\cite{RAMM504}. In this paper the constant regularization parameter
$a$ in iterative scheme \eqref{itr1} is replaced with the geometric
series $\{\alpha_0q^n\}_{n=1}^\infty,$ $\alpha_0>0,\ q\in(0,1)$,
i.e. \be\label{itr2}
u_{n+1}^\dl=\alpha_0q^nT_{\alpha_0q^n}^{-1}u_n^\dl+T_{\alpha_0q^n}^{-1}A^*f_\dl.
\ee Stopping rule \eqref{srule} (see below) is used for this
iterative scheme. Without loss of generality we use $\alpha_0=1$ in
\eqref{itr2}. The convergence analysis of this iterative scheme is
presented in Section 3. In Section 4 some numerical experiments are
given to illustrate the efficiency of the proposed methods.
\section{Derivation of the proposed method }
In this section we give a detailed derivation of iterative schemes
\eqref{itnf1} and \eqref{itn1}. Let us denote by $y\in\R^m$ the
unique minimal-norm solution of equation (\ref{11}). Throughout this
paper we denote $T_{a(t)}:=A^*A+a(t)I_m$, where $I_m$ is the
identity operator in $\R^m$, and $a(t)$ is given in \eqref{ae}.
\begin{lem}\label{lemq} Let $g(x)$ be a continuous function on
$(0,\infty)$, $c>0$ and $q\in(0,1)$ be constants. If \be\label{g}
\lim_{x\to 0^+}g(x)=g(0):=g_0,\ee then \be \lim_{n\to
\infty}\sum_{j=1}^{n-1}\left(q^{n-j-1}-q^{n-j}\right)g(cq^{j+1})=
g_0. \ee
\end{lem} \begin{proof}
Let \be\label{omg}\omega_j(n):=q^{n-j}-q^{n+1-j},\quad
\omega_j(n)>0,\ee and
\be\label{Fn}F_l(n):=\sum_{j=1}^{l-1}\omega_{j}(n)g(cq^{j}).\ee
Then \bee |F_{n+1}(n)-g_0|\leq |F_{l}(n)|+\left|\sum_{j=l}^n
\omega_j(n)g(cq^{j})-g_0\right|.\eee Take $\epsilon>0$ arbitrary
small. For sufficiently large $l(\epsilon)$ one can choose
$n(\epsilon)$, such that \bee |F_{l(\epsilon)}(n)|\leq
\frac{\epsilon}{2},\ \forall n>n(\epsilon), \eee because
$\lim_{n\to\infty}q^n=0.$ Fix $l=l(\epsilon)$ such that
$|g(cq^j)-g_0|\leq \frac{\epsilon}{2}$ for $j>l(\epsilon)$. This is
possible because of \eqref{g}. One has \bee |F_{l(\epsilon)}(n)|\leq
\frac{\epsilon}{2},\ n>n(\epsilon) \eee and \bee\begin{split}
\left|\sum_{j=l(\epsilon)}^n \omega_j(n)g(cq^{j})-g_0\right|&\leq
\sum_{j=l(\epsilon)}^{n}
\omega_j(n)|g(cq^{j})-g_0|+|\sum_{j=l(\epsilon)}^{n}
\omega_j(n)-1||g_0|\\
&\leq
\frac{\epsilon}{2}\sum_{j=l(\epsilon)}^n\omega_j(n)+q^{n-l(\epsilon)}|g_0|\\
&\leq \frac{\epsilon}{2}+|g_0|q^{n-l(\epsilon)}\leq
\epsilon,\end{split}\eee if $n$ is sufficiently large. Here we have
used the relation\bee \sum_{j=l}^{n}\omega_j(n)=1-q^{n+1-l}. \eee
Since $\epsilon>0$ is arbitrarily small, \lemref{lemq} is proved.

\end{proof}

Let us define \be\label{ME}
u_n:=\int_0^{t_n}e^{-(t_n-s)}T_{a^{(n)}(s)}^{-1}A^*fds,\quad
t_n=-n\ln(q), \quad q\in(0,1).\ee Note that \bee\begin{split}
u_n&=\int_0^{t_{n-1}}e^{-(t_n-s)}T_{a^{(n)}(s)}^{-1}A^*fds+\int_{t_{n-1}}^{t_{n}}e^{-(t_n-s)}T_{a^{(n)}(s)}^{-1}A^*fds\\
&=e^{-(t_n-t_{n-1})}\int_0^{t_{n-1}}e^{-(t_{n-1}-s)}T_{a(s)}^{-1}A^*fds+\int_{t_{n-1}}^{t_{n}}e^{-(t_n-s)}T_{a^{(n)}(s)}^{-1}A^*fds\\
&=e^{-(t_n-t_{n-1})}u_{n-1}+\int_{t_{n-1}}^{t_{n}}e^{-(t_n-s)}T_{a^{(n)}(s)}^{-1}A^*fds.
\end{split}\eee
Using definition \eqref{ae}, one gets \bee\begin{split}
u_n&=e^{-(t_n-t_{n-1})}u_{n-1}+[1-e^{-(t_n-t_{n-1})}]T_{\alpha_0q^{n}}^{-1}A^*f\\
&=\frac{q^{n}}{q^{n-1}}u_{n-1}+(1-\frac{q^{n}}{q^{n-1}})T_{\alpha_0q^{n}}^{-1}A^*f.
\end{split}\eee Therefore, \eqref{ME} can be rewritten as iterative scheme
\eqref{itnf1}.
\begin{lem}\label{lemex}
Let $u_n$ be defined in \eqref{itnf1} and $Ay=f$. Then \be\label{UB}
\|u_n-y\|\leq q^n\|y\|+\sum_{j=0}^{n-1}\left(
q^{n-j-1}-q^{n-j}\right)\alpha_0q^{j+1}\|T_{\alpha_0q^{j+1}}^{-1}y\|,\quad
\forall n\geq 1,\ee and \be\label{conu} \|u_n-y\|\to 0 \text{ as }
n\to \infty. \ee
\end{lem}
\begin{proof}
By definitions \eqref{ME} and \eqref{ae} we obtain
\be\begin{split}\label{dun}
u_n=\int_0^{t_n}e^{-(t_n-s)}T_{a(s)}^{-1}A^*fds=\sum_{j=0}^{n-1}\left(\frac{q^{n}}{q^{j+1}}-\frac{q^{n}}{q^j}\right)T_{\alpha_0q^{j+1}}^{-1}A^*f.
\end{split}\ee From \eqref{dun} and the equation $Ay=f$, one gets: \bee\begin{split}
u_{n}&=\sum_{j=0}^{n-1}\left(\frac{q^{n}}{q^{j+1}}-\frac{q^{n}}{q^j}\right)T_{\alpha_0q^{j+1}}^{-1}A^*f\\
&=\sum_{j=0}^{n-1}\left(\frac{q^{n}}{q^{j+1}}-\frac{q^{n}}{q^j}\right)T_{\alpha_0q^{j+1}}^{-1}(T_{\alpha_0q^{j+1}}-\alpha_0q^{j+1}I_m)y\\
&=\sum_{j=0}^{n-1}\left(q^{n-j-1}-q^{n-j}\right)y-\sum_{j=0}^{n-1}\left(q^{n-j-1}-q^{n-j}\right)\alpha_0q^{j+1}T_{\alpha_0q^{j+1}}^{-1}y\\
&=y-q^ny-\sum_{j=0}^{n-1}\left(q^{n-j-1}-q^{n-j}\right)\alpha_0q^{j+1}T_{\alpha_0q^{j+1}}^{-1}y.
\end{split}\eee Thus, estimate \eqref{UB} follows.
To prove \eqref{conu}, we apply \lemref{lemq} with
$g(a):=a\|T_a^{-1}y\|.$ Since $y\perp \mathcal{N}(A)$, it follows
from the spectral theorem that $$\lim_{a\to 0}g^2(a)=\lim_{a\to 0}
\int_0^\infty\frac{a^2}{(a+s)^2}d\langle
E_sy,y\rangle=\|P_{\mathcal{N}(A)}y\|^2=0,$$ where $E_s$ is the
resolution of the identity corresponding to $A^*A$, and $P$ is the
orthogonal projector onto $\mathcal{N}(A)$. Thus, by \lemref{lemq},
\eqref{conu} follows.
\end{proof}

Let us discuss iterative scheme \eqref{itn1}. The following lemma
gives the estimate of the difference of the solutions $u_n^\dl$ and
$u_n$.
\begin{lem}\label{lemudu} Let $u_n$ and $u_n^\dl$ be defined in
\eqref{itnf1} and \eqref{itn1}, respectively. Then \be\label{udu}
\|u_{n}^\dl-u_n\|\leq \frac{\sqrt{q}}{1-q^{3/2}}w_n,\quad n\geq
0,\ee where
$w_n:=(1-q)\frac{\delta}{2\sqrt{q}\sqrt{\alpha_0q^{n}}}.$
\end{lem}
\begin{proof} Let $H_n:=\|u_{n}^\dl-u_{n}\|$. Then from the definitions of $u_n^\dl$ and $u_n$ we get the estimate \be H_{n+1}\leq
q\|u_n^\dl-u_n\|+(1-q)\|T_{\alpha_0q^{n+1}}^{-1}A^*(f_\dl-f)\|\leq
qH_n+w_n. \ee Let us prove inequality \eqref{udu} by induction. For
$n=0$ one has $u_0=u_0^\dl=0$, so \eqref{udu} holds. For $n=1$ one
has $\|u_1^\dl-u_1\|\leq (1-q)\frac{\dl}{2\sqrt{\alpha_0q^2}}$, so
\eqref{udu} holds. If \eqref{udu} holds for $n\leq k$, then for
$n=k+1$ one has \be
\begin{split} H_{k+1}&\leq qH_k+w_k\leq
\left(\frac{q^{3/2}}{1-q^{3/2}}+1\right)w_k=\frac{1}{1-q^{3/2}}w_k\\
&=\frac{1}{1-q^{3/2}}\frac{w_k}{w_{k+1}}w_{k+1}\leq
\frac{1}{1-q^{3/2}}\sqrt{q}w_{k+1}.\end{split}\ee Hence \eqref{udu}
is proved for $n\geq 0$.
\end{proof}

\subsection{Stopping criterion } In this section we give a stopping
rule for iterative scheme given in \eqref{itn1}. Let $Q:=AA^*$,
$Q_{a}:=Q+aI_m,$ and\be\label{GN}\begin{split}
G_n&:=\int_0^{t_n}e^{-(t_n-s)}a(s)\|Q_{a(s)}^{-1}f_\dl\|ds\\
&=\sum_{j=0}^{n-1}(q^{n-j-1}-q^{n-j})\alpha_0q^{j+1}\|Q_{\alpha_0q^{j+1}}^{-1}f_\dl\|,\quad
n\geq 1, \end{split}\ee where $t_n=-n\ln q,$ $q\in(0,1)$ and
$\alpha_0>0.$ Then stopping rule \eqref{rule} can be rewritten as
\be\label{rule1} G_{n_\dl}\leq C\dl^\varepsilon<G_{n},\quad 1\leq
n<n_\dl,\quad \varepsilon\in(0,1),\quad C>1,\quad
G_1>C\dl^\varepsilon. \ee Note that \bee\begin{split}
G_{n+1}&=\sum_{j=0}^{n}(q^{n-j}-q^{n+1-j})\alpha_0q^{j+1}\|Q_{\alpha_0q^{j+1}}^{-1}f_\dl\|\\
&=\sum_{j=0}^{n-1}(q^{n-j}-q^{n+1-j})\alpha_0q^{j+1}\|Q_{\alpha_0q^{j+1}}^{-1}f_\dl\|+(1-q)\alpha_0q^{n+1}\|Q_{\alpha_0q^{n+1}}^{-1}f_\dl\|\\
&=qG_{n}+(1-q)\alpha_0q^{n+1}\|Q_{\alpha_0q^{n+1}}^{-1}f_\dl\|,
\end{split}\eee so
\be\label{Gn}
G_n=qG_{n-1}+(1-q)\alpha_0q^{n}\|Q_{\alpha_0q^{n}}^{-1}f_\dl\|,\quad
n\geq 1,\quad G_0=0. \ee

\begin{lem}\label{lemda}
Let $G_n$ be defined in \eqref{Gn}. Then \be\label{Gn2d} G_n\leq
\frac{1}{1-\sqrt{q}}\sqrt{\alpha_0q^n}\frac{\|y\|}{2}+\dl ,\quad
n\geq 1,\quad q\in(0,1).\ee
\end{lem}
\begin{proof}
Using the identity \bee -aQ_a^{-1}=AT_a^{-1}A^*-I_m,\ a>0,\
T:=A^*A,\ T_a:=T+aI_m,\eee the estimates \bee a\|Q_a^{-1}\|\leq 1,\
\|f_\dl-f\|\leq \dl, \eee and \bee a\|AT_a^{-1}\|\leq
\frac{\sqrt{a}}{2}, \eee where $Q:=AA^*$, $Q_a:=Q+aI_m$, we get
\be\begin{split}\label{Ges1} G_{n}&=qG_{n-1}+(1-q)\|AT_{\alpha_0q^{n}}^{-1}A^*f_\dl-f_\dl\|\\
&=qG_{n-1}+(1-q)\|AA^*Q_{\alpha_0q^{n}}^{-1}f_\dl-f_\dl\|\\
&=qG_{n-1}+(1-q)\|(AA^*+\alpha_0q^{n}I-\alpha_0q^nI)Q_{q^{n}}^{-1}f_\dl-f_\dl\|\\
&=qG_{n-1}+(1-q)\alpha_0q^n\|Q_{\alpha_0q^{n}}^{-1}f_\dl\|\\
&=qG_{n-1}+(1-q)\alpha_0q^n\|Q_{\alpha_0q^{n}}^{-1}(f_\dl-f+f)\|\\
&\leq qG_{n-1}+(1-q)\alpha_0q^n\|Q_{\alpha_0q^{n}}^{-1}(f_\dl-f)\|+(1-q)\alpha_0q^n\|Q_{\alpha_0q^{n}}^{-1}f\|\\
&\leq qG_{n-1}+(1-q)\dl+(1-q)\|AT_{\alpha_0q^n}^{-1}A^*f-f\|\\
&=qG_{n-1}+(1-q)\dl+(1-q)\|A(T_{\alpha_0q^n}^{-1}A^*Ay-y)\|\\
&=qG_{n-1}+(1-q)\dl+(1-q)\|A(-\alpha_0q^nT_{\alpha_0q^n}^{-1}y)\|\\
&=qG_{n-1}+(1-q)\dl+(1-q)\alpha_0q^n\|AT_{\alpha_0q^n}^{-1}y\|\\
&\leq qG_{n-1}+(1-q)\dl+(1-q)\alpha_0q^n\frac{\|y\|}{2\sqrt{\alpha_0q^n}}\\
&=qG_{n-1}+(1-q)\dl+(1-q)\sqrt{\alpha_0q^n}\frac{\|y\|}{2}\\
&=qG_{n-1}+(1-q)\dl+
\sqrt{q}\frac{\sqrt{\alpha_0q^{n-1}}}{2}\|y\|.\end{split}\ee
Therefore,
\be\label{Gnd} G_n-\dl\leq
q(G_{n-1}-\dl)+\sqrt{q}\frac{\sqrt{\alpha_0q^{n-1}}}{2}\|y\|,\quad
n\geq 1,\ G_0=0. \ee Let us prove relation \eqref{Gn2d} by
induction. From relation \eqref{Gnd} we get \be G_1-\dl\leq
-q\dl+\frac{\sqrt{\alpha_0q}}{2}\|y\|\leq-q\dl+
\frac{1}{1-\sqrt{q}}\frac{\sqrt{\alpha_0q}}{2}\|y\| \leq
\frac{1}{1-\sqrt{q}}\frac{\sqrt{\alpha_0q}}{2}\|y\|. \ee Thus, for
$n=1$ relation \eqref{Gn2d} holds. Suppose that \be\label{Gind1}
G_{n}-\dl\leq
\frac{1}{1-\sqrt{q}}\frac{\sqrt{\alpha_0q^{n}}}{2}\|y\|,\quad 1\leq
n\leq k. \ee Then by inequalities \eqref{Gnd} and \eqref{Gind1} we
obtain \be\begin{split} G_{k+1}-\dl&\leq
q(G_k-\dl)+\sqrt{q}\frac{\sqrt{\alpha_0q^{k}}}{2}\|y\|\\
&\leq
q\frac{1}{1-\sqrt{q}}\frac{\sqrt{\alpha_0q^{k}}}{2}\|y\|+\sqrt{q}\frac{\sqrt{\alpha_0q^{k}}}{2}\|y\|\\
&=
\frac{\sqrt{q}}{1-\sqrt{q}}\frac{\sqrt{\alpha_0q^{k}}}{2}\|y\|=\frac{\sqrt{q}}{1-\sqrt{q}}\frac{\sqrt{\alpha_0q^{k}}}{2\sqrt{\alpha_0q^{k+1}}}\sqrt{\alpha_0q^{k+1}}\|y\|\\
&\leq
\frac{1}{1-\sqrt{q}}\frac{\sqrt{\alpha_0q^{k+1}}}{2}\|y\|.\end{split}\ee
Thus, relation \eqref{Gn2d} is proved.
\end{proof}

\begin{lem}\label{monGn}
Let $G_n$ be defined in \eqref{Gn}, $q\in(0,1)$, and $\alpha_0>0$ be
chosen such that $G_1>C\dl^\varepsilon$, $\varepsilon\in(0,1)$,
$C>1$. Then there exists a unique integer $n_c$ such that \be
G_{n_c-1}<G_{n_c} \text{ and }G_{n_c}>G_{n_c+1},\quad n_c\geq 1. \ee
Moreover,\be\label{GGnc} G_{n+1}<G_n,\quad \forall n\geq n_c.\ee
\end{lem}
\begin{proof}
From \lemref{lemda} we have $$ G_n\leq
\frac{1}{1-\sqrt{q}}\sqrt{\alpha_0q^n}\frac{\|y\|}{2}+\dl ,\quad
n\geq 1,\quad q\in(0,1).$$ Since $G_1>C\dl^\varepsilon$ and
$\limsup_{n\to \infty}G_n\leq \dl<C\dl^\varepsilon$, it follows that
there exists an integer $n_c\geq 1$ such that $G_{n_c-1}<G_{n_c}$
and $G_{n_c}>G_{n_c+1}.$ Let us prove the monotonicity of $G_n$, for
$n\geq n_c$. We have $G_{n_c+1}-G_{n_c}<0$. Using definition
\eqref{Gn}, we get \be\begin{split}
G_{n_c+1}-G_{n_c}&=qG_{n_c}+(1-q)\alpha_0q^{n_c+1}\|Q_{\alpha_0q^{n_c+1}}^{-1}f_\dl\|-G_{n_c}\\
&=(1-q)\left(\alpha_0q^{n_c+1}\|Q_{\alpha_0q^{n_c+1}}^{-1}f_\dl\|-G_{n_c}\right)<0.
\end{split}\ee This implies \be\label{AsG}
\alpha_0q^{n_c+1}\|\alpha_0Q_{q^{n_c+1}}^{-1}f_\dl\|-G_{n_c}<0. \ee
Note that \bee
G_{n+1}-G_n=(1-q)(\alpha_0q^{n+1}\|Q_{\alpha_0q^{n+1}}^{-1}f_\dl\|-G_{n}).\eee
Therefore, to prove the monotonicity of $G_n$ for $n\geq n_c$, one
needs to prove the inequality
$$\alpha_0q^{n+1}\|Q_{\alpha_0q^{n+1}}^{-1}f_\dl\|-G_{n}<0,\quad\forall n\geq n_c.$$
This inequality is a consequence of the following lemma:
\begin{lem}\label{lemqG}
Let $G_n$ be defined in \eqref{Gn}, and \eqref{AsG} holds. Then
\be\label{QGn}
\alpha_0q^{n+1}\|Q_{\alpha_0q^{n+1}}^{-1}f_\dl\|-G_{n}<0,\quad
\forall n\geq n_c. \ee
\end{lem}
\begin{proof}Let us prove \lemref{lemqG} by induction.
Let
$$D_n:=\alpha_0q^{n+1}\|Q_{\alpha_0q^{n+1}}^{-1}f_\dl\|-G_n$$ and
$$h(a):=a^2\|Q_a^{-1}f_\dl\|^2.$$ The function $h(a)$ is a monotonically growing function of $a$, $a>0$. Indeed, by the spectral
theorem, we get
 \be\begin{split} h(a_1)&=a_1^2\|Q_{a_1}^{-1}f_\dl\|^2=\int_0^\infty \frac{a_1^2}{(a_1+s)^2}d\langle
F_sf_\dl,f_\dl\rangle\\
&\leq\int_0^\infty \frac{a_2^2}{(a_2+s)^2}d\langle
F_sf_\dl,f_\dl\rangle=a_2^2\|Q_{a_2}^{-1}f_\dl\|^2=h(a_2),
\end{split}\ee where $F_s$ is the resolution of the identity
corresponding to $Q:=AA^*$, because
$\frac{a_1^2}{(a_1+s)^2}\leq\frac{a_2^2}{(a_2+s)^2}$ if $0<a_1<a_2$
and $s\geq0$.
 By the assumption we have
$D_{n_c}=\alpha_0q^{n_c+1}\|Q_{\alpha_0q^{n_c+1}}^{-1}f_\dl\|-G_{n_c}<0.$
Thus, relation \eqref{QGn} holds for $n=n_c.$ For $n=n_c+1$ we get
\be\begin{split}
D_{n_c+1}&=\alpha_0q^{n_c+2}\|Q_{\alpha_0q^{n_c+2}}^{-1}f_\dl\|-(1-q)\alpha_0q^{n_c+1}\|Q_{\alpha_0q^{n_c+1}}^{-1}f_\dl\|-qG_{n_c}\\
&=\sqrt{h(\alpha_0q^{n_c+2})}-\sqrt{h(\alpha_0q^{n_c+1}
)}+q\sqrt{h(\alpha_0q^{n_c+1}
)}-qG_{n_c}\\
&=\sqrt{h(\alpha_0q^{n_c+2})}-\sqrt{h(\alpha_0q^{n_c+1}
)}+q(\sqrt{h(\alpha_0q^{n_c+1}
)}-G_{n_c})\\
&=\sqrt{h(\alpha_0q^{n_c+2})}-\sqrt{h(\alpha_0q^{n_c+1}
)}+qD_{n_c}\\
&=\sqrt{h(\alpha_0q^{n_c+2})}-\sqrt{h(\alpha_0q^{n_c+1}
)}+qD_{n_c}<0.\end{split}\ee

Here we have used the monotonicity of the function $h(a).$ Thus,
relation \eqref{QGn} holds for $n=n_c+1.$ Suppose \bee D_n<0,\quad
n_c\leq n\leq n_c+k-1.\eee This, together with the monotonically
growth of the function $h(a):=a^2\|Q_q^{-1}f_\dl\|^2$, yields
\be\begin{split}
D_{n_c+k}&=\alpha_0q^{n_c+k+1}\|Q_{\alpha_0q^{n_c+k+1}}^{-1}f_\dl\|-G_{n_c+k}\\
&=\sqrt{h(\alpha_0q^{n_c+k+1})}-(1-q)\sqrt{h(\alpha_0q^{n_c+k})}-qG_{n_c+k-1}\\
&=\sqrt{h(\alpha_0q^{n_c+k+1})}-\sqrt{h(\alpha_0q^{n_c+k})}+q(\sqrt{h(\alpha_0q^{n_c+k})}-G_{n_c+k-1})\\
&=\sqrt{h(\alpha_0q^{n_c+k+1})}-\sqrt{h(\alpha_0q^{n_c+k})}+qD_{n_c+k-1}\\
&=\sqrt{h(\alpha_0q^{n_c+k+1})}-\sqrt{h(\alpha_0q^{n_c+k})}+qD_{n_c+k-1}<0.
\end{split}\ee
Thus, $D_n<0,\ n\geq 1.$ \lemref{lemqG} is proved.
\end{proof}
Let us continue with the proof of \lemref{monGn}. From relation
\eqref{Gn} we have \bee\begin{split}
G_{n+1}-G_n&=(q-1)G_n+(1-q)\alpha_0q^{n+1}\|Q_{\alpha_0q^{n+1}}^{-1}f_\dl\|\\
&=(1-q)\left(\alpha_0q^{n+1}\|Q_{\alpha_0q^{n+1}}^{-1}f_\dl\|-G_n\right).
\end{split}\eee Using assumption \eqref{AsG} and applying
\lemref{lemqG}, one gets \bee G_{n+1}-G_n<0,\quad \forall n\geq n_c.
\eee Let us prove that the integer $n_c$ is unique. Suppose there
exists another integer $n_d$ such that $G_{n_d-1}<G_{n_d}$ and
$G_{n_d}>G_{n_d+1}$. One may assume without loss of generality that
$n_c<n_d.$ Since $G_{n}>G_{n+1},\ \forall n\geq n_c$, and $n_c<n_d$,
it follows that $G_{n_d-1}>G_{n_d}$. This contradicts the assumption
$G_{n_d-1}<G_{n_d}.$ Thus, the integer $n_c$ is unique.
\lemref{monGn} is proved.
\end{proof}

\begin{lem}\label{Gnmon}
Let $G_n$ be defined in \eqref{Gn}. If $\alpha_0$ is chosen such
that relations $G_1>C\dl^\varepsilon$, $C>1$, $\varepsilon\in(0,1)$,
holds then there exists a unique $n_\dl$ satisfying inequality
\eqref{rule1}.
\end{lem}
\begin{proof}Let us show that there exists an integer
$n_\dl$ so that inequality \eqref{rule1} holds. Applying
\lemref{lemda}, one gets \be \limsup_{n\to \infty}G_n\leq \dl. \ee
Since $G_1>C\dl^\varepsilon$ and $\limsup_{n\to \infty}G_n\leq
\dl<C\dl^\varepsilon$, it follows that there exists an index $n_\dl$
satisfying stopping rule \eqref{rule1}. The uniqueness of the index
$n_\dl$ follows from the monotonicity of $G_n$, see \lemref{monGn}.
Thus, \lemref{Gnmon} is proved.
\end{proof}
\begin{lem}\label{lemand}Let $Ay=f$, $y \perp \mathcal{N}(A)$, and $n_\dl$ be chosen by rule \eqref{rule1}.
Then \be\label{dqn} \lim_{\dl\to 0}q^{n_\dl}= 0,\quad q\in(0,1), \ee
so \be\label{ndl} \lim_{\dl\to 0}n_\dl= \infty.\ee
\end{lem}\begin{proof}
From rule \eqref{rule1} and relation \eqref{Gn} we have
\be\begin{split}
qC\dl^\varepsilon+(1-q)\alpha_0q^{n_\dl}\|Q_{\alpha_0q^{n_\dl}}^{-1}f_\dl\|&<qG_{n_\dl-1}+(1-q)\alpha_0q^{n_\dl}\|Q_{\alpha_0q^{n_\dl}}^{-1}f_\dl\|\\
&=G_{n_\dl}\leq C\dl^\varepsilon,
\end{split}\ee so
\be (1-q)\alpha_0q^{n_\dl}\|Q_{\alpha_0q^{n_\dl}}^{-1}f_\dl\|\leq
(1-q)C\dl^\varepsilon. \ee Thus, \be\label{uze}
\alpha_0q^{n_\dl}\|Q_{\alpha_0q^{n_\dl}}^{-1}f_\dl\|<C\dl^\varepsilon.
\ee Note that if $f\neq 0$ then there exists a $\lambda_0>0$ such
that \be\label{Flo} F_{\lambda_0}f\neq 0,\quad \langle
F_{\lambda_0}f,f\rangle:=\xi>0, \ee where $\xi$ is a constant which
does not depend on $\dl$, and $F_s$ is the resolution of the
identity corresponding to the operator $Q:=AA^*$. Let
$$h(\dl,\alpha):=\alpha^2\|Q_{\alpha}^{-1}f_\dl\|^2.$$ For a
fixed number $c_1>0$ we obtain \be\begin{split}
h(\dl,c_1)&=c_1^2\|Q_{c_1}f_\dl\|^2=\int_0^\infty
\frac{c_1^2}{(c_1+s)^2}d\langle F_sf_\dl,f_\dl\rangle\geq
\int_0^{\lambda_0} \frac{c_1^2}{(c_1+s)^2}d \langle
F_sf_\dl,f_\dl\rangle\\
&\geq \frac{c_1^2}{(c_1+\lambda_0)^2}\int_0^{\lambda_0} d\langle
F_sf_\dl,f_\dl\rangle=\frac{c_1^2\|F_{\lambda_0}f_\dl\|^2}{(c_1+\lambda_0)^2},\quad
\dl>0.
\end{split}\ee
Since $F_{\lambda_0}$ is a continuous operator, and
$\|f-f_\dl\|<\dl$, it follows from \eqref{Flo} that \be \lim_{\dl\to
0}\langle F_{\lambda_0}f_\dl,f_\dl\rangle=\langle
F_{\lambda_0}f,f\rangle>0 .\ee Therefore, for the fixed number
$c_1>0$ we get \be\label{hc1} h(\dl,c_1)\geq c_2>0\ee for all
sufficiently small $\dl>0$, where $c_2$ is a constant which does not
depend on $\dl$. For example one may take $c_2=\frac{\xi}{2}$
provided that \eqref{Flo} holds. Let us derive from estimate
\eqref{uze} and the relation \eqref{hc1} that $q^{n_\dl}\to 0$ as
$\dl\to 0$. From \eqref{uze} we have
$$0\leq h(\dl,\alpha_0q^{n_\dl})\leq (C\dl^\varepsilon)^2.$$ Therefore,
\be\label{limhq}\lim_{\dl\to 0} h(\dl, \alpha_0q^{n_\dl})=0.\ee
Suppose $\lim_{\dl\to 0}q^{n_\dl}\neq 0$. Then there exists a
subsequence $\dl_j\to 0$ such that \be \alpha_0q^{n_{\dl_j}}\geq
c_1>0, \ee where $c_1$ is a constant. By \eqref{hc1} we
get\be\label{hdj} h(\dl_j,\alpha_0q^{n_{\dl_j}})>c_2>0,\quad
\dl_j\to 0\text{ as }j\to \infty. \ee This contradicts relation
\eqref{limhq}. Thus, $\lim_{\dl\to0} q^{n_\dl}=0.$ \lemref{lemand}
is proved.
\end{proof}

\begin{lem}\label{lemdan}
Let $n_\dl$ be chosen by rule \eqref{rule1}. Then \be\label{dan}
\frac{\dl}{\sqrt{\alpha_0q^{n_\dl}}}\to 0\text{ as }\dl\to 0.\ee
\end{lem}\begin{proof} Relation \eqref{Gn2d}, together with stopping rule
\eqref{rule1}, implies \be C\dl^\varepsilon<G_{n_\dl-1}\leq
\frac{1}{1-\sqrt{q}}\frac{\sqrt{\alpha_0q^{n_\dl-1}}}{2}\|y\|+\dl.
\ee Then \be \frac{1}{\sqrt{\alpha_0q^{n_\dl-1}}}\leq
\frac{\|y\|}{2(1-\sqrt{q})\dl^\varepsilon(C-1) },\quad
\varepsilon\in(0,1). \ee This yields \be \lim_{\dl\to
0}\frac{\dl}{\sqrt{\alpha_0q^{n_\dl}}}=\lim_{\dl\to
0}\frac{\dl}{\sqrt{\alpha_0qq^{n_\dl-1}}}\leq \lim_{\dl\to
0}\frac{\dl^{1-\varepsilon}}{2\sqrt{q}(1-\sqrt{q})(C-1)}\|y\|=0. \ee
\lemref{lemdan} is proved.  \end{proof}

\begin{thm}\label{mthm}
Let $y\perp \mathcal{N}(A)$, $\|f_\dl-f\|\leq \dl$,
$\|f_\dl\|>C\dl^\varepsilon,$ $C>1,$ $\varepsilon \in(0,1).$ Suppose
$n_\dl$ is chosen by rule \eqref{rule1}. Then \be\label{limuy}
\lim_{\dl\to 0}\|u_{n_\dl}^\dl-y\|=0 ,\ee where $u_n^\dl$ is given
in \eqref{itn1}.
\end{thm}
\begin{proof}
Using \lemref{lemex} and \lemref{lemudu}, we get the estimate
\be\begin{split}\label{convuy} \|u_{n_\dl}^\dl-y\|&\leq
\|u_{n_\dl}^\dl-u_{n_\dl}\|+\|u_{n_\dl}-y\|\leq
\frac{\sqrt{q}}{1-q^{3/2}}(1-q)\frac{\delta}{2q\sqrt{\alpha_0q^{n_\dl}}}+\|u_{n_\dl}-y\|\\
&:=I_1+I_2,
  \end{split}\ee where
$I_1:=\frac{\sqrt{q}}{1-q^{3/2}}(1-q)\frac{\delta}{2q\sqrt{\alpha_0q^{n_\dl}}}$
and $I_2:=\|u_{n_\dl}-y\|.$ Applying \lemref{lemdan}, one gets
$\lim_{\dl\to 0}I_1=0.$ Since $n_\dl\to \infty$ as $\dl\to 0$, it
follows from \lemref{lemex} that $\lim_{\dl\to 0} I_2=0.$ Thus,
$\lim_{\dl\to 0}\|u_{n_\dl}^\dl-y\|=0.$ \thmref{mthm} is proved.
\end{proof}
The algorithm based on the proposed method can be stated as follows:
\begin{itemize}
\item[Step 1.] Assume that \eqref{fdf} holds. Choose $C\in(1,2)$ and $\varepsilon\in(0.9,1)$. Fix
$q\in(0,1)$, and choose $\alpha_0>0$ so that \eqref{Asfd} holds.
Set $n=1$, and $u_0=0$.
\item[Step 2.] Use iterative scheme \eqref{itn1} to calculate $u_n$.
\item[Step 3.] Calculate $G_n$, where $G_n$ is defined in \eqref{Gn}.
\item[Step 4.] If $G_n\leq C\dl^\varepsilon$ then stop the iteration, set
$n_\dl=n$, and take $u_{n_\dl}^\dl$ as the approximate solution.
Otherwise set $n=n+1$, and go to Step 1.
\end{itemize}
\section{Iterative scheme 2}
In \cite{RAMM504} the following iterative scheme for the exact data
$f$ is given: \be\label{ItF} u_{n+1}=aT_a^{-1}u_n+T_a^{-1}A^*f,\quad
u_1=u_1\perp \mathcal{N}(A),\ee where $a$ is a fixed positive
constant. It is proved in \cite{RAMM504} that iterative scheme
\eqref{ItF} gives the relation
$$\lim_{n\to \infty}\|u_n-y\|=0,\quad y\perp \mathcal{N}(A).$$ In
the case of noisy data the exact data $f$ in \eqref{ItF} is replaced
with the noisy data $f_\dl$, i.e. \be\label{ItFn}
u_{n+1}^\dl=aT_a^{-1}u_n^\dl+T_a^{-1}A^*f_\dl,\quad u_1=u_1\perp
\mathcal{N}(A),\ee where $\|f_\dl-f\|\leq \dl$ for sufficiently
small $\dl>0$. It is proved in \cite{RAMM504} that there exist an
integer $n_\dl$ such that \be \lim_{\dl\to 0}\|u_{n_\dl}^\dl-y\|=0,
\ee where $u_n^\dl$ is the approximate solution corresponds to the
noisy data. But a method of choosing the integer $n_\dl$ has not
been discussed. In this section we modify iterative scheme
\eqref{ItF} by replacing the constant parameter $a$ in \eqref{ItF}
with a geometric sequence $\{q^{n-1}\}_{n=1}^\infty$, i.e.
\be\label{iram}
u_{n+1}=q^{n}T_{q^{n}}^{-1}u_n+T_{q^{n}}^{-1}A^*f,\quad u_1=0,\ee
where $q\in(0,1)$. The initial approximation $u_1$ is chosen to be
0. In general one may choose an arbitrary initial approximation
$u_1$ in the set $\mathcal{N}(A)^\perp.$ If the data are noisy then
the exact data $f$ in \eqref{iram} is replaced with the noisy data
$f_\dl$, and iterative scheme \eqref{ItFn} is replaced with:
\be\label{iram2}
u_{n+1}^\dl=q^{n}T_{q^{n}}^{-1}u_n^\dl+T_{q^{n}}^{-1}A^*f_\dl,\quad
u_1^\dl=0.\ee We prove convergence of the solution obtained by
iterative scheme \eqref{iram} in \thmref{thmnfcr} for arbitrary
$q\in(0,1)$, i.e. \bee \lim_{n\to \infty}\|u_n-y\|=0,\quad \forall
q\in(0,1). \eee In the case of noisy data we use discrepancy-type
principle \eqref{srule} to obtain the integer $n_\dl$ such that
\be\label{cno} \lim_{\dl\to 0}\|u_{n_\dl}^\dl-y\|=0. \ee We prove relation \eqref{cno}, for arbitrary $q\in(0,1)$, in \thmref{Mthm2}.\\
Let us prove that the sequence $u_n$, defined by iterative scheme
\eqref{iram}, converges to the minimal norm solution $y$ of equation
\eqref{11}.
\begin{thm}\label{thmnfcr}
Consider iterative scheme \eqref{iram}. Let $y\perp \mathcal{N}(A)$.
Then \be\label{ciram} \lim_{n\to \infty}\|u_n-y\|=0. \ee
\end{thm}
\begin{proof}
Consider the identity \be\label{iy} y=aT_a^{-1}y+T_a^{-1}A^*f,\quad
Ay=f.\ee Let $w_n:=u_n-y$ and $B_n:=q^{n}T_{q^{n}}^{-1}.$ Then
$w_{n+1}=B_nw_n,\quad w_1=y-u_1=y$. One uses \eqref{iy} and gets
\be\begin{split}\label{res} \|y-u_n\|^2&=\|B_{n-1}B_{n-2}\hdots
B_1w_1\|^2=\|B_{n-1}B_{n-2}\hdots B_1y\|^2\\
&=\int_0^\infty \left(\frac{q^{n-1}}{q^{n-1}+s}
\frac{q^{n-2}}{q^{n-2}+s} \hdots
\frac{q}{q+s} \right)^2d\langle E_sy,y\rangle\\
&=\int_0^\infty \left(\frac{q^{n-1}}{q^{n-1}+s}
\right)^2\left(\frac{q^{n-2}}{q^{n-2}+s} \right)^2\hdots
\left(\frac{q}{q+s} \right)^2d\langle E_sy,y\rangle\\
&\leq \int_0^\infty \frac{q^{2n}}{(q+s)^{2n}}d\langle
E_sy,y\rangle,\end{split}\ee where $E_s$ is the resolution of the
identity corresponding to the operator $T:=A^*A$. Here we have used
the identity \eqref{iy} and the monotonicity of the function
$\phi(x):=\frac{x^2}{(x+s)^2},\ s\geq 0.$ From estimate \eqref{res}
we derive relation \eqref{ciram}. Indeed, write \be \int_0^\infty
\frac{q^{2n}}{(q+s)^{2n}}d\langle E_sy,y\rangle=\int_0^b
\frac{q^{2n}}{(q+s)^{2n}}d\langle
E_sy,y\rangle+\int_b^\infty\frac{q^{2n}}{(q+s)^{2n}}d\langle
E_sy,y\rangle,\ee where $b$ is a sufficiently small number which
will be chosen later. For any fixed $b>0$ one has
$\frac{q}{q+s}\leq\frac{q}{q+b}<1$ if $s\geq b$. Therefore, it
follows that \be \int_b^\infty\frac{q^{2n}}{(q+s)^{2n}}d\langle
E_sy,y\rangle\to 0 \text{ as }n\to \infty. \ee On the other hand one
has \be \int_0^b\frac{q^{2n}}{(q+s)^{2n}}d\langle E_sy,y\rangle\leq
\int_0^b d\langle E_sy,y\rangle . \ee Since $y \perp
\mathcal{N}(A),$ one has $\lim_{b\to 0}\int_0^b d\langle
E_sy,y\rangle=0.$ Therefore, given an arbitrary number $\epsilon>0$
one can choose $b(\epsilon)$ such that \be
\int_0^{b(\epsilon)}\frac{q^{2n}}{(q+s)^{2n}}d\langle
E_sy,y\rangle<\frac{\epsilon}{2}. \ee Using this $b(\epsilon)$, one
chooses sufficiently large $n(\epsilon)$ such that \be
\int_{b(\epsilon)}^\infty\frac{q^{2n}}{(q+s)^{2n}}d\langle
E_sy,y\rangle<\frac{\epsilon}{2},\quad \forall n>n(\epsilon). \ee
Since $\epsilon>0$ is arbitrary, \thmref{thmnfcr} is proved.
\end{proof}
As we mentioned before if the exact data $f$ are contaminated by
some noise then iterative scheme \eqref{iram2} is used, where
$\|f_\dl-f\|\leq \dl$. Note that \be\begin{split}\label{udyr}
\|u_{n+1}^\dl-u_{n+1}\|&\leq
q^{n}\|T_{q^{n}}^{-1}(u_n^\dl-u_n)\|+\frac{\dl}{2\sqrt{q^{n}}}\leq
\|u_n^\dl-u_n\|+\frac{\dl}{2\sqrt{q^{n}}}.
\end{split}\ee To prove the
convergence of the solution obtained by iterative scheme
\eqref{iram2}, we need the following lemmas:

\begin{lem}\label{lemr1} Let $u_n$ and $u_n^\dl$ be defined in \eqref{iram} and \eqref{iram2}, respectively. Then
\be\label{esudu} \|u_n^\dl-u_n\|\leq
\frac{\sqrt{q}}{1-\sqrt{q}}\frac{\dl}{2\sqrt{q^{n}}},\quad n\geq
1.\ee \end{lem}
\begin{proof} Let us prove relation \eqref{esudu} by induction. For
$n=1$ one has $u_1^\dl-u_1=0$. Thus, for $n=1$ the relation holds.
Suppose \be\label{in1} \|u_{l}^\dl-u_l\|\leq
\frac{\sqrt{q}}{1-\sqrt{q}}\frac{\dl}{2\sqrt{q^{l}}},\quad 1\leq
l\leq k. \ee Then from \eqref{udyr} and \eqref{in1} we have
\bee\begin{split} \|u_{k+1}^\dl-u_{k+1}\|&\leq
\|u_k^\dl-u_k\|+\frac{\dl}{2\sqrt{q^{k}}}\leq
\frac{\sqrt{q}}{1-\sqrt{q}}\frac{\dl}{2\sqrt{q^{k}}}+\frac{\dl}{2\sqrt{q^{k}}}\\
&\leq\sqrt{q}\frac{\dl}{(1-\sqrt{q})2\sqrt{q^{k+1}}}.
\end{split}\eee Thus, $$\|u_{n}^\dl-u_{n}\|\leq \sqrt{q}\frac{\dl}{(1-\sqrt{q})2\sqrt{q^{n}}},\quad n\geq 1 .$$
\lemref{lemr1} is proved.
\end{proof}

Let us formulate our stopping rule: the iteration in iterative
scheme \eqref{iram2} is stopped at the first integer $n_\dl$
satisfying \be\label{srule}
\|AT_{q^{n_\dl}}^{-1}A^*f_\dl-f_\dl\|\leq
C\dl^\varepsilon<\|AT_{q^{n}}^{-1}A^*f_\dl-f_\dl\|,\quad 1\leq
n<n_\dl,\ C>1,\ \varepsilon\in(0,1), \ee and it is assumed that
$\|f_\dl\|>C\dl^\varepsilon.$
\begin{lem}\label{monWn}
Let $u_n^\dl$ be defined in \eqref{iram2}, and
$W_n:=\|AT_{q^n}^{-1}A^*f_\dl-f_\dl\|$. Then \be W_{n+1}\leq
W_n,\quad n\geq 1. \ee
\end{lem}
\begin{proof} Note that \be\label{Wnn}W_n=\|AA^*Q_{q^n}^{-1}f_\dl-f_\dl\|=\|(Q_{q^n}-q^nI_m)Q_{q^n}^{-1}f_\dl-f_\dl\|=\|q^nQ_{q^n}^{-1}f_\dl\|,\ee
where $Q:=AA^*$, and $Q_a:=Q+aI_m.$ Using the spectral theorem, one
gets \be\label{esWW}
W_{n+1}^2=\int_0^\infty\frac{q^{2(n+1)}}{(q^{n+1}+s)^2}d\langle F_s
f_\dl,f_\dl\rangle\leq \int_0^\infty\frac{q^{2n}}{(q^n+s)^2}d\langle
F_s f_\dl,f_\dl\rangle=W_n^2, \ee where $F_s$ is the resolution of
the identity corresponding to the operator $Q:=AA^*$. Here we have
used the monotonicity of the function
$g(x)=\frac{x^2}{(x+s)^2},\quad s\geq 0.$ Thus, \be W_{n+1}\leq
W_n,\quad n\geq 1. \ee \lemref{monWn} is proved.
\end{proof}

\begin{lem}
Let $u_n^\dl$ be defined in \eqref{iram2}, and
$\|f_\dl\|>C\dl^\varepsilon$, $\varepsilon\in(0,1)$, $C>1$. Then
there exists a unique index $n_\dl$ such that inequality
\eqref{srule} holds.
\end{lem}
\begin{proof}
Let $e_n:=AT_{q^n}^{-1}A^*f_\dl-f_\dl$. Then \be
e_{n}=q^{n}Q_{q^{n}}^{-1}f_\dl, \ee where $Q_a:=AA^*+aI.$ Therefore,
\be\begin{split}\label{en} \|e_{n}\|&\leq
\|q^{n}Q_{q^{n}}^{-1}(f_\dl-f)\|+\|q^{n}Q_{q^{n}}^{-1}f\|\\
&\leq \|f_\dl-f\|+\|q^{n}Q_{q^{n}}^{-1}Ay\|\leq
\dl+\frac{\sqrt{q^{n}}}{2}\|y\|,
\end{split}\ee where the estimate $\|Q_a^{-1}A\|=\|AT_a^{-1}\|\leq \frac{1}{2\sqrt{a}}$ was used. Thus, $$\limsup_{n\to \infty}\|e_n\|\leq
\dl.$$ This shows that the integer $n_\dl$, satisfying
\eqref{srule}, exists. The uniqueness of $n_\dl$ follows from its
definition.
\end{proof}

\begin{lem}\label{lemr2}
Let $u_n^\dl$ be defined in \eqref{iram2}. If $n_\dl$ is chosen by
rule \eqref{srule} then \be \lim_{\dl\to
0}\frac{\dl}{\sqrt{q^{n_\dl}}}=0. \ee
\end{lem}
\begin{proof}
From \eqref{en} we have \be\begin{split}\label{esr2}
\|AT_{q^{n-1}}^{-1}A^*f_\dl-f_\dl\|&\leq
\dl+\frac{\sqrt{q^{n-1}}}{2}\|e_1\| ,\end{split}\ee where
$e_1:=u_1-y=-y$. It follows from stopping rule \eqref{srule} and
estimate \eqref{esr2} that \be\begin{split} C\dl^\varepsilon&\leq
\|AT_{q^{n_\dl-1}}^{-1}A^*f_\dl-f_\dl\|\leq
\frac{\sqrt{q^{n_\dl-1}}}{2}\|e_1\|+\dl.
\end{split}\ee Therefore,
\be \left(C-1\right)\dl^\varepsilon\leq
\frac{\sqrt{q^{n_\dl-1}}}{2}\|e_1\|, \ee and so \be
\frac{1}{\sqrt{q^{n_\dl-1}}}\leq
\frac{\|e_1\|}{2(C-1)\dl^{\varepsilon}},\quad \varepsilon\in(0,1).
\ee This implies \be
\frac{\dl}{\sqrt{q^{n_\dl}}}=\frac{\dl}{\sqrt{qq^{n_\dl-1}}}\leq
\frac{\|e_1\|\dl}{2q^{1/2}(C-1)\dl^{\varepsilon}}=\frac{\|e_1\|}{2q^{1/2}(C-1)}\dl^{1-\varepsilon}.\ee
Thus, $\frac{\dl}{\sqrt{q^{n_\dl}}}\to 0$ as $\dl\to 0.$
\lemref{lemr2} is proved.
\end{proof}
The proof of convergence of the solution obtained by iterative
scheme \eqref{iram2} is given in the following theorem:
\begin{thm}\label{Mthm2}
Let $u_n^\dl$ be defined in \eqref{iram2}, $y\perp \mathcal{N}(A)$,
$\|f_\dl\|>C\dl^\varepsilon$, $\varepsilon\in(0,1)$, $C>1,\quad
q\in(0,1)$. If $n_\dl$ is chosen by rule \eqref{srule}, then \be
\|u_n^\dl-y\|\to 0 \text{ as } \dl\to 0. \ee
\end{thm}
\begin{proof}
From \lemref{lemr1} we get the following estimate: \be
\|u_{n_\dl}^\dl-y\|\leq
\|u_{n_\dl}^\dl-u_{n_\dl}\|+\|u_{n_\dl}-y\|\leq
\frac{\sqrt{q}}{1-\sqrt{q}}\frac{\dl}{2\sqrt{q^{n_\dl}}}+\|u_{n_\dl}-y\|:=I_1+I_2,
\ee where
$I_1:=\frac{\sqrt{q}}{1-\sqrt{q}}\frac{\dl}{2\sqrt{q^{n_\dl}}}$ and
$I_2:=\|u_{n_\dl}-y\|.$ By \lemref{lemr2} one gets $I_1\to 0$ as
$\dl\to 0.$ To prove $\lim_{\dl\to 0}I_2=0$ one needs the relation
$\lim_{\dl\to 0} n_\dl=\infty$. This relation is a consequence of
the following lemma:
\begin{lem}\label{lemnr}
If $n_\dl$ is chosen by rule \eqref{srule}, then \be q^{n_\dl}\to 0
\text{ as } \dl\to 0, \ee so \be \lim_{\dl\to 0} n_\dl=\infty.\ee
\end{lem}
\begin{proof}
Note that \bee\begin{split}
AT_{a}^{-1}A^*f_\dl-f_\dl&=AA^*Q_{a}^{-1}f_\dl-f_\dl=(AA^*+aI_m-aI_m)Q_a^{-1}f_\dl-f_\dl\\
&=f_\dl-aQ_a^{-1}f_\dl-f_\dl=-aQ_a^{-1}f_\dl,\end{split}\eee where
$a>0$, $Q:=AA^*$ and $Q_a:=Q+aI.$ From stopping rule \eqref{srule}
we have $0\leq \|AT_{q^{n_\dl}}^{-1}A^*f_\dl-f_\dl\|\leq
C\dl^\varepsilon$.  Thus, \be \lim_{\dl\to
0}\|AT_{q^{n_\dl}}^{-1}A^*f_\dl-f_\dl\|=\lim_{\dl\to 0}
\|q^{n_\dl}Q_{q^{n_\dl}}^{-1}f_\dl\|=0.\ee Using an argument given
in the proof of \lemref{lemand}, (see formulas
\eqref{uze}-\eqref{hdj} in which $\alpha_0=1$), one gets
$\lim_{\dl\to 0}q^{n_\dl}=0$, so $\lim_{\dl\to 0}n_\dl=\infty.$
\lemref{lemnr} is proved.
\end{proof}
\lemref{lemnr} and \thmref{thmnfcr} imply $I_2\to 0$ as $\dl\to 0.$
Thus, \thmref{Mthm2} is proved.
\end{proof}

\section{Numerical experiments}
In all the experiments we measure the accuracy of the approximate
solutions using the relative error: \bee \text{Rel.Err
}=\frac{\|u_{n_\dl}^\dl-y\|}{\|y\|}, \eee where $\|.\|$ is the
Euclidean norm in $\R^n$. The exact data are perturbed by some
noises so that \bee \|f_\dl-f\|\leq\dl , \eee where \bee
f_\dl=f+\dl\frac{e}{\|e\|}, \eee $\dl$ is the noise level, and $e\in
\R^n$ is the noise taken from the Gaussian distribution with mean
$0$ and standard deviation 1. The MATLAB routine called "randn" with
seed 15 is used to generate the vector $e$. The iterative schemes
\eqref{itn1} and \eqref{iram2} will be denoted by $IS_1$ and $IS_2$,
respectively. In the iterative scheme $IS_1$, for fixed $q\in(0,1)$,
one needs to choose a sufficiently large $\alpha_0>0$ so that
inequality \eqref{Asfd} hold, for example one may choose
$\alpha_0\geq 1.$ The number of iterations of $IS_1$ and $IS_2$ are
denoted by $Iter_1$ and $Iter_2$, respectively. We compare the
results obtained by the proposed methods with the results
obtained by using the
variational regularization method (VR). In VR we use the Newton
method for solving the equation for regularization parameter.
In
\cite{MRZ84} the nonlinear equation
\be\label{DPVR} \|Au_{VR}(a)-f_\dl\|^2=(C\dl)^2,\ C=1.01,
\ee
where $u_{VR}(a):=T_{a}^{-1}A^*f_\dl,$
is solved by the Newton's method. In this paper the initial
value of the regularization parameter is taken to be
$\alpha_0=\frac{\alpha_0}{2^{k_\dl}}$,
where $k_\dl$ is the first integer such that the Newton's method for
solving \eqref{DPVR} converges. We stop the iteration of the
Newton's method at the first integer $n_\dl$ satisfying the
inequality $|\|AT_{a_n}^{-1}A^*f_\dl-f_\dl\|^2-(C\dl)^2|\leq 10^{-3}
(C\dl)^2,\quad a_0:=\alpha_0.$ The number of iterations needed
to complete a convergent Newton's method is denoted by $Iter_{VR}$.
\subsection{Ill-conditioned linear algebraic systems}
\begin{table}[h]\caption{Condition number of some Hilbert matrices}\label{conmat}
\newcommand{\m}{\hphantom{$-$}}
\renewcommand{\tabcolsep}{2pc} 
\renewcommand{\arraystretch}{1.2} 
\begin{tabular}[htb]{ll}
\hline
$n$&$\kappa(A)=\|A\|\|A^{-1}\|$\\
\hline
$10$& $1.915\times 10^{13}$ \\
$20$& $1.483\times 10^{28}$\\
$70$&$8.808\times 10^{105}$\\
$100$& $1.262\times 10^{150}$\\
$200$&$1.446\times 10^{303}$ \\
\hline
\end{tabular}\\[2pt]
\end{table}

Consider the following system: \be\label{h1} H^{(m)}u=f, \ee where
\bee H^{(m)}_{ij}=\frac{1}{i+j+1},\quad i,j=1,2,...,m,\eee is a
Hilbert matrix of dimension $m$. The system \eqref{h1} is an example
of a severely ill-posed problem if $m>10$, because the condition
number of the Hilbert matrix is increasing exponentially as $m$
grows, see Table \eqref{conmat}. The minimal eigenvalues of Hilbert
matrix of dimension $m$ can be obtained using the following formula
\be\label{minlam}
\lambda_{\min}(H^{(m)})=2^{15/4}\pi^{3/2}\sqrt{m}(\sqrt{2}+1)^{-(4m+4)}(1+o(1)).
\ee
 This formula is proved in
\cite{GAK01}. Since $\kappa
(H^{(m)})=\frac{\lambda_{max}(H^{(m)})}{\lambda_{min}(H^{(m)})}$, it
follows from \eqref{minlam} that the condition number grows as
$O(\frac{e^{3.5255m}}{\sqrt{m}})$.
The following exact solution is used to test the proposed methods:\\
$$y\in \R^m, \text{ where } y_k=\sqrt{.5 k},\
k=1,2,\hdots,m.$$ The Hilbert matrix of dimension $m=200$ is used in
the experiments. This matrix has condition number of order
$10^{303}$, so it is a severely ill-conditioned matrix.
\begin{table}[ht]\caption{Hilbert matrix problem: the number of iterations and the relative errors with respect to the parameter $q$ ($\alpha_0=1,\quad \dl=10^{-2}$).  }
\newcommand{\m}{\hphantom{$-$}}
\renewcommand{\tabcolsep}{.85pc} 
\renewcommand{\arraystretch}{1.2} 
\begin{tabular}{lllll}
\hline
&\multicolumn{2}{c}{}\\
\hline
$q$&\multicolumn{2}{c}{$IS_1$}&\multicolumn{2}{c}{$IS_2$}\\
\hline
 & REl.Err&$Iter_1$& REl.Err&$Iter_2$.\\
\hline

$.5$&$  0.031$&$  24$&$  0.032 $&$ 23 $\\

$.25$&$0.031$&$  13$&$  0.032$&$  13$\\
$.125$&$  0.032$&$  9$&$  0.032$&$9$\\
\hline
\end{tabular}
\end{table}
In Table 2 one can see that the number of iterations of the
iterative scheme $IS_1$ and $IS_2$ increases as the value of $q$
increases. The relative errors start to increase at $q=.125$. By
these observations, we suggest to choose the parameter $q$ in the
interval $(.125,.5).$ In Table 3 the results of the experiments with
various values of $\dl$ are presented.
\begin{table}[ht]\caption{ICLAS with Hilbert matrix: the relative errors and the number of iterations   }
\newcommand{\m}{\hphantom{$-$}}
\renewcommand{\tabcolsep}{.85pc} 
\renewcommand{\arraystretch}{1.2} 
\begin{tabular}{lllllll}
\hline
&\multicolumn{3}{c}{}\\
\hline
$\delta$&\multicolumn{2}{c}{$IS_1$}&\multicolumn{2}{c}{$IS_2$}&\multicolumn{2}{c}{VR}\\
\hline
 & REl.Err&$Iter_1$& REl.Err&$Iter_2$& REl.Err&$Iter_{VR}$.\\
\hline
$5\%$&$0.038$&$  11 $&$ 0.043 $&$11$&$  0.055 $&$13
$\\
$3\%$&$0.037$&$  12 $&$ 0.034 $&$ 12 $&$ 0.045 $&$14  $\\
$1\%$&$ 0.031$&$  13 $&$ 0.032 $&$ 13$&$  0.034$&$15 $ \\
\hline
\end{tabular}
\end{table}
Here the parameter $\varepsilon$ was .99. The geometric sequence
$\{.25^{n-1}\}_{n=1}^\infty$ was used in the iterative schemes
$IS_1$ and $IS_2$. The parameter $C$ in \eqref{rule} and
\eqref{srule} were $1.01$. The parameter $k_\dl$ in the variational
regularization method was 1. One can see that the relative errors of
$IS_1$ and $IS_2$ are smaller than these for  the VR. The relative
error decreases as the noise level decreases which can be seen on
the same table. This shows that the proposed method produces stable
solutions.

\subsection{Fredholm integral equations of the first kind (FIEFK)}
 Here we consider two Fredholm integral equations :
\begin{itemize}
\item[a)]
\be\label{Phi} f(s)=\int_{-3}^3k(t-s)u(t)dt, \ee where \be
k(z)=\left\{
                                             \begin{array}{ll}
                                               1 + \cos(\frac{\pi}{3}z) , &\hbox{$|z| <  3$;} \\
                                               0, & \hbox{$|z|\geq 3$,}
                                             \end{array}
                                           \right.
 \ee
and \be f(z)=\left\{
               \begin{array}{ll}
                 (6+z)\left[1-\frac{1}{2}\cos(\frac{\pi}{3}z)\right]-\frac{9}{2\pi}\sin(\frac{\pi z}{3}), & \hbox{$|z|\leq 6$;} \\
                 0, & \hbox{$|z|>6$.}
               \end{array}
             \right.
 \ee
\item[b)]\be\label{Derv} f(s)=\int_{0}^1k(s,t)u(t)dt, s\in(0,1),\ee where \be
k(s,t)=\left\{
                                             \begin{array}{ll}
                                               s(t-1) , &\hbox{$s<t$;} \\
                                               t(s-1), & \hbox{$s\geq t$,}
                                             \end{array}
                                           \right. \ee and \be
f(s)=(s^3 - s)/6 .\ee

\end{itemize}
The problem $a)$ is discussed in \cite{PHI62} where the solution to
this problem is $u(x)=k(x)$. The second problem is taken from
\cite{LMJL85} where the solution is $u(x)=x$. The Galerkin's method
is used to discretized the integrals \eqref{Phi} and \eqref{Derv}.
For the basis functions we use the following orthonormal box
functions \be \phi_i(s):=\left\{
                          \begin{array}{ll}
                            \sqrt{\frac{m}{c_1}}, & \hbox{$[s_{i-1},s_i]$ ;} \\
                            0, & \hbox{otherwise,}
                          \end{array}
                        \right.\ee and
\be \psi_i(t):=\left\{
                          \begin{array}{ll}
                            \sqrt{\frac{m}{c_2}}, & \hbox{$[t_{i-1},t_i]$ ;} \\
                            0, & \hbox{otherwise,}
                          \end{array}
                        \right.\ee where $s_i=d_1+i\frac{d_2}{m}$, $t_i=d_3+i\frac{d_4}{m},\ i=0,1,2,\hdots,m.$ In the problem $a)$ the
                        parameters $c_1,$ $c_2$, $d_1,\ d_2,\ d_3$ and $d_4$ are set to $12$, $6$, $-6,\ 12,\ -3$ and $6$,
respectively. In the second problem we use $d_1=d_3=0$ and
$c_1=c_2=d_2=d_4=1$. Here we approximate the solution $u(t)$ by $
\tilde{u}=\sum_{j=1}^mc_j\psi_j(t).$  Therefore solving problem
\eqref{Phi} is reduced to solving the linear algebraic system \be
A\tilde{c}=f,\ \tilde{c},f\in \R^m, \ee where in problem $a)$
$$A_{ij}=\int_{-3}^{3}\int_{-6}^6k(t-s)\phi_i(s)\psi_j(t)dsdt$$ and
$f_i=\int_{-6}^6f(s)\phi_i(s)ds,\ i,j=1,2,\hdots,m$, and in problem
$b)$
$$A_{ij}=\int_{0}^{1}\int_{0}^1k(s,t)\phi_i(s)\psi_j(t)dsdt$$ and
$f_i=\int_{0}^1f(s)\phi_i(s)ds,$ and $\tilde{c}_j=c_j\
i,j=1,2,\hdots,m$.

\begin{table}[ht]\caption{Problem $a)$: the
number of iterations and the relative errors with respect to the
parameter $q$ ($\alpha_0=2,\quad \dl=10^{-2}$).
}\newcommand{\m}{\hphantom{$-$}}
\renewcommand{\tabcolsep}{.85pc} 
\renewcommand{\arraystretch}{1.2} 
\begin{tabular}{lllll}
\hline
&\multicolumn{2}{c}{}\\
\hline
$q$&\multicolumn{2}{c}{$IS_1$}&\multicolumn{2}{c}{$IS_2$}\\
\hline
 & REl.Err&$Iter_1$& REl.Err&$Iter_2$.\\
\hline

$.5$&$  0.008$&$  12$&$  0.007 $&$ 11 $\\

$.25$&$0.009 $&$ 7 $&$ 0.007$&$  7$\\
$.125$&$  0.009$&$5$&$  0.008$&$ 5
$\\
\hline
\end{tabular}
\end{table}

\begin{table}[htp]\caption{Problem $a)$: the relative errors and the number of iterations   }
\newcommand{\m}{\hphantom{$-$}}
\renewcommand{\tabcolsep}{.85pc} 
\renewcommand{\arraystretch}{1.2} 
\begin{tabular}{lllllll}
\hline
&\multicolumn{3}{c}{}\\
\hline
$\delta$&\multicolumn{2}{c}{$IS_1$}&\multicolumn{2}{c}{$IS_2$}&\multicolumn{2}{c}{VR}\\
\hline
 & REl.Err&$Iter_1$& REl.Err&$Iter_2$& REl.Err&$Iter_{VR}.$\\
\hline
$   5\%    $&$0.018$&$ 6 $&$ 0.014$&$  6 $&$ 0.016$&$ 11
$ \\
$   3\%   $&$ 0.013$&$  6$&$  0.011$&$  6$&$  0.013 $&$12
$\\
$   1\%  $&$0.009 $&$ 7$&$  0.007 $&$ 7 $&$ 0.008$&$ 15 $\\
\hline
\end{tabular}
\end{table}

\begin{table}[ht]\caption{Problem $b)$: the number of iterations and the relative errors with respect to the parameter $q$ ($\alpha_0=4,\quad\dl=10^{-2}$).  }
\newcommand{\m}{\hphantom{$-$}}
\renewcommand{\tabcolsep}{.85pc} 
\renewcommand{\arraystretch}{1.2} 
\begin{tabular}{lllll}
\hline
&\multicolumn{2}{c}{}\\
\hline
$q$&\multicolumn{2}{c}{$IS_1$}&\multicolumn{2}{c}{$IS_2$}\\
\hline
 & REl.Err&$Iter_1$& REl.Err&$Iter_2$.\\
\hline

$.5$&$  0.428 $&$ 17 $&$ 0.446$&$  15 $\\

$.25$&$0.421 $&$ 9 $&$ 0.436 $&$ 9$\\
$.125$&$  0.439 $&$ 6$&$  0.416 $&$7
$\\
\hline
\end{tabular}
\end{table}

\begin{table}[htp]\caption{Problem $b)$: the relative errors and the number of iterations   }
\newcommand{\m}{\hphantom{$-$}}
\renewcommand{\tabcolsep}{.85pc} 
\renewcommand{\arraystretch}{1.2} 
\begin{tabular}{lllllll}
\hline
&\multicolumn{3}{c}{}\\
\hline
$\delta$&\multicolumn{2}{c}{$IS_1$}&\multicolumn{2}{c}{$IS_2$}&\multicolumn{2}{c}{VR}\\
\hline
 & REl.Err&$Iter_1$& REl.Err&$Iter_2$& REl.Err&$Iter_{VR}.$\\
\hline
$   5\%    $&$
  0.618 $&$ 7 $&$ 0.621 $&$ 7 $&$ 0.627$&$12$ \\
$   3\%   $&$ 0.541 $&$ 8 $&$ 0.559 $&$ 8$&$
0.584$&$13 $\\
$   1\%  $&$0.421 $&$ 9$&$  0.436 $&$ 9$&$
0.457 $&$13 $\\
\hline
\end{tabular}
\end{table}

The parameter $m=600$ is used in problem $a)$. In this case the
condition number of the matrix $A$ with $m=600$ is $3.427 \times
10^9$, so it is an ill-conditioned matrix. Here the parameter $C$ in
$IS_1$ and $IS_2$ are $2$ and $1.01$, respectively. For problem $b)$
the parameter $m$ is 200. In this case the condition number of the
matrix $A$ is $4.863\times 10^4$. The parameter $C$ is $1.01$ in the
both iterative schemes $IS_1$ and $IS_2$. In Tables 4 and 6 we give
the relation between the parameter $q$ and the number of iterations
and the relative errors of the iterative schemes $IS_1$ and $IS_2$.
The closer the parameter $q$ to 1, the larger number of iterations
we get, and the closer the parameter $q$ to 0, the smaller the
number of iterations we get. But the relative error starts to
increase if the parameter $q$ is chosen too small. Based on the
numerical results given in Tables 4 and 5, we suggest to choose the
parameter $q$ in the interval $(0.125,0.5)$. In the iterative
schemes $IS_1$ and $IS_2$ we use the geometric sequence $\{2\times
.25^{n-1}\}_{n=1}^\infty$ for problem a). The geometric series
$\{4\times .25^{n-1}\}_{n=1}^\infty$ is used in problem b). In the
variational regularization method we use $\alpha_0=2$ and
$\alpha_0=4$ as the initial regularization parameter of the Newton's
method in problem $a)$ and $b)$, respectively. Since the Newton's
method for solving \eqref{DPVR} is locally convergent, in problem
$b)$ we need to choose a smaller regularization parameter $\alpha_0$
than for $IS_1$ and $IS_2$ methods. Here $k_\dl=8$ was used. The
numerical results on Table 5 show that the solutions produced by the
proposed iterative schemes are stable. In problem $a)$ the relative
errors of the iterative scheme $IS_2$, are smaller than these for
the iterative scheme $IS_1$ and than these for the variational
regularization, VR. In Table 7 the relative errors produced by the
three methods for solving problem $b)$ are presented. The relative
error of $IS_1$ is smaller than the one for the other two methods.

\section{Conclusion} We have demonstrated that the proposed iterative
schemes can be used for solving ill-conditioned linear algebraic systems
stably. The advantage of the iterative scheme \eqref{itn1} compared with
iterative scheme \eqref{iram2} is the following: one applies the operator
$T_a^{-1}$ only once at each iteration. Note that the difficulty of using
the Newton's method is in choosing the initial value for the
regularization parameter, since the Newton's method for solving equation
\eqref{DPVR} converges only locally. In solving \eqref{DPVR} by the
Newton's method one often has to choose an initial regularization
parameter $a_0$ sufficiently close to the root of equation \eqref{DPVR} as
shown in problem $b)$ in Section 4.2. In our iterative schemes the initial
regularization parameter can be chosen in the interval $[1,4]$ which is
larger than the initial regularization parameter used in the variational
regularization method. In the iterative scheme $IS_1$ we modified the
discrepancy-type principle
$$\int_0^te^{-(t-s)}a(s)\|Q_{a(s)}^{-1}\|ds=C\dl,\ C\in(1,2),$$ given in
\cite{RAMM525}, by using \eqref{ae} to get discrepancy-type principle
\eqref{rule1}, which can be easily implemented numerically. In Section 3
we used the geometric series $\{\alpha_0 q^n\}_{n=1}^\infty$ in place of
the constant regularization parameter $a$ in the iterative scheme
$$u_{n+1}=aT_a^{-1}u_n+T_a^{-1}A^*f_\dl$$ developed in \cite{RAMM504}.
This geometric series of the regularization parameter allows one to use
the a posteriori stopping rule given in \eqref{srule}. We proved that this
stopping rule produces stable approximation of the minimal norm solution
of equation \eqref{11}. In all the experiments stopping rules
\eqref{rule1} and \eqref{srule} produce stable approximations to the
minimal norm solution of equation \eqref{11}.
It is of interest to develop a
method for choosing the parameter $q$ in the proposed methods which gives
sufficiently small relative error and small number of iterations.\\

\end{document}